\documentclass[10pt]{amsart}
\usepackage[english]{babel}
\usepackage{amssymb}
\usepackage{amsmath}
\usepackage{srcltx}
\usepackage{lscape}

\newtheorem{dummy}{Dummy}

\newtheorem{theorem}[dummy]{Theorem}
\newtheorem{proposition}[dummy]{Proposition}
\newtheorem{corollary}[dummy]{Corollary}

\theoremstyle{definition}
\newtheorem{definition}{Definition}

\newtheorem{example}[dummy]{Example}

\newtheorem{remark}[dummy]{Remark}

\newcommand{\ignore}[1]{}

\author{S. Pumpl\"un}
\email{susanne.pumpluen@nottingham.ac.uk}
\address{School of Mathematical Sciences\\
University of Nottingham\\
University Park\\
Nottingham NG7 2RD\\
United Kingdom
}



\begin{document}

\title[Algebras whose right nucleus is a central simple algebra]
{Algebras whose right nucleus is a central simple algebra}

\begin{abstract}
We generalize Amitsur's construction of central simple algebras over a field $F$
which are split by field extensions  possessing a derivation with field of constants $F$  to nonassociative algebras:
 for every central division algebra $D$ over a field $F$ of characteristic zero there exists an
 infinite-dimensional unital nonassociative algebra  whose right nucleus is
$D$ and whose left and middle nucleus are a field extension $K$ of $F$ splitting $D$, where $F$ is algebraically closed in
$K$.

We then give a short direct proof that every $p$-algebra of degree $m$, which has a purely inseparable splitting field
$K$ of degree $m$ and exponent one, is a differential extension of $K$ and cyclic.
We  obtain finite-dimensional division
 algebras over a field $F$ of characteristic $p>0$ whose right nucleus is a division $p$-algebra.
\end{abstract}

\maketitle

%
\section*{Introduction}
%

In 1954, Amitsur \cite{Am} observed that all associative central division algebras over a field $F$ of characteristic zero can
be constructed using differential polynomials. His construction method can be considered as an analogue
to the the well known crossed product construction,
except that he uses  splitting fields $K$ of the algebras, where the base field $F$ is algebraically closed in $K$,
instead of their algebraic splitting fields.
Some of his results also work for  $p$-algebras, i.e. over base fields of characteristic $p>0$.

In this paper, we consider  algebras which are
also obtained from differential polynomials, but which are nonassociative.

These algebras are constructed using the differential polynomial ring $K[t;\delta]$, where $K$ is a field and $\delta$
a derivation on $K$ and were defined by Petit \cite{P66}: given a differential polynomial $f\in K[t;\delta]$ of degree $m$,
the set of all differential polynomials of degree less than $m$, together with the
addition given by the usual addition of polynomials, can be equipped with a nonassociative
ring structure using right division by $f$ to define the multiplication as $g\circ h=gh \,\,{\rm mod}_r f $.
The resulting nonassociative unital ring $S_f$, also denoted by $K[t;\delta]/K[t;\delta]f$,
is an algebra over the field of constants $F={\rm Const}(\delta)$ of $\delta$. If $f$ generates
 a two-sided ideal in $K[t;\delta]$, then $S_f$ is the (associative)
quotient algebra obtained by factoring out the two-sided principal ideal generated by $f$.

If $f$ is not two-sided and $\delta$ not trivial, then the nuclei of $S_f$ are larger than the center
$F={\rm Const}(\delta)$. In that case the left and middle nucleus are always given by $K$, whereas
the right nucleus reflects both the choice of $f$ and the structure of the ring $K[t;\delta]$.

 We proceed as follows:
The basic terminology and notation we use can be found in \cite{Am} and Section \ref{sec:prel}.
Section 2 rephrases some of Amitsur's results for those algebras $S_f$ which have a central simple algebra as their right nucleus.
For this we employ Amitsur's $A$-polynomials.
In Sections 3 and 4 we show how to construct algebras $S_f$ with a given central simple algebra as right nucleus,
first for base fields of characteristic zero, then for base fields of characteristic $p>0$:
for every central simple algebra $B$ of degree $m$ over a field $F$ of characteristic zero which
is split by a field extension $K/F$ in which $F$ is algebraically closed, there exists
 an infinite-dimensional unital algebra $S_f=K[t;\delta]/K[t;\delta]f$ over $F$ with right nucleus $B$
  (and left and middle nucleus $K$),
see Theorem \ref{thm:mainI}. In particular, for every  central division algebra $D$ over $F$ there exists an
infinite-dimensional unital algebra $S_f$ over $F$ with right nucleus $D$  (Corollary \ref{cor:mainI}).

 We present a short  proof that
every $p$-algebra $B$ of degree $m$ over a field $F$ of characteristic $p$
which is split by a purely inseparable field extension $K/F$ of exponent one and degree $m$ is isomorphic to a differential
extension $(K,\delta,d_0)$ of $K$ (Theorem  \ref{thm:mammone}),
 only invoking a result on the structure of $S_f$ and Amitsur's  \cite[Lemma 20']{Am}. Thus it is cyclic by \cite[Main Theorem]{H}.

For every division $p$-algebra $D$ of degree $m$ over a field $F$ of characteristic $p$
which is split by a purely inseparable field extension $K/F$ of exponent one such that
$m<[K:F]$, there is a unital division algebra $S_f=K[t;\delta]/K[t;\delta]f$ over $F$ of dimension $m p^e$
 with right nucleus $D$ and left and middle nucleus $K$.
The smallest possible dimension $l$ of such a division algebra
containing $D$ as right nucleus is bounded via $m^2< l\leq mp^{m-1}$ and
 connected to the number of cyclic algebras that
 are needed when expressing $D$ as a product of cyclic algebras of degree $p$ in the Brauer group $Br(F)$
 (Corollary \ref{cor:last}).

%
%

\section{Preliminaries} \label{sec:prel}

\subsection{Nonassociative algebras} \label{subsec:nonassalgs}


Let $F$ be a field and let $A$ be an $F$-vector space. $A$ is an
\emph{algebra} over $F$ if there exists an $F$-bilinear map $A\times
A\to A$, $(x,y) \mapsto x \cdot y$, denoted simply by juxtaposition
$xy$, the  \emph{multiplication} of $A$. An algebra $A$ is called
\emph{unital} if there is an element in $A$, denoted by 1, such that
$1x=x1=x$ for all $x\in A$. We will only consider unital algebras
from now on without explicitly saying so.

An algebra $A\not=0$ is called a \emph{division algebra} if for any
$a\in A$, $a\not=0$, the left multiplication  with $a$, $L_a(x)=ax$,
and the right multiplication with $a$, $R_a(x)=xa$, are bijective.
If $A$ has finite dimension over $F$, $A$ is a division algebra if
and only if $A$ has no zero divisors \cite[pp. 15, 16]{Sch}.

Associativity in $A$ is measured by the {\it associator} $[x, y, z] =
(xy) z - x (yz)$. The {\it left nucleus} of $A$ is defined as ${\rm
Nuc}_l(A) = \{ x \in A \, \vert \, [x, A, A]  = 0 \}$, the {\it
middle nucleus} of $A$ is ${\rm Nuc}_m(A) = \{ x \in A \, \vert \, [A, x, A]  = 0 \}$ and  the {\it right nucleus} of $A$ as
${\rm Nuc}_r(A) = \{ x \in A \, \vert \, [A,A, x]  = 0 \}$. ${\rm Nuc}_l(A)$, ${\rm Nuc}_m(A)$, and ${\rm Nuc}_r(A)$ are associative
subalgebras of $A$. Their intersection
 ${\rm Nuc}(A) = \{ x \in A \, \vert \, [x, A, A] = [A, x, A] = [A,A, x] = 0 \}$ is the {\it nucleus} of $A$.
${\rm Nuc}(A)$ is an associative subalgebra of $A$ containing $F1$
and $x(yz) = (xy) z$ whenever one of the elements $x, y, z$ is in
${\rm Nuc}(A)$. The
 {\it center} of $A$ is ${\rm C}(A)=\{x\in \text{Nuc}(A)\,|\, xy=yx \text{ for all }y\in A\}$.


\subsection{Differential polynomial rings}

Let $K$ be a field and $\delta:K\rightarrow K$ a \emph{derivation}, i.e. an
additive map such that $$\delta(ab)=a\delta(b)+\delta(a)b$$ for all $a,b\in K$.
 The \emph{differential polynomial ring} $K[t;\delta]$
is the set of polynomials $$a_0+a_1t+\dots +a_nt^n$$ with $a_i\in K$,
where addition is defined term-wise and multiplication by
$$ta=at+\delta(a) \quad (a\in K).$$
For $f=a_0+a_1t+\dots +a_nt^n$ with $a_n\not=0$ define ${\rm
deg}(f)=n$ and ${\rm deg}(0)=-\infty$. Then ${\rm deg}(fg)={\rm deg}
(f)+{\rm deg}(g).$
 An element $f\in R$ is \emph{irreducible} in $R$ if it is not a unit and if it has no proper factors,
 i.e if there do not exist $g,h\in R$ with
 ${\rm deg}(g),{\rm deg} (h)<{\rm deg}(f)$ such
 that $f=gh$.

 $R=K[t;\delta]$ is a left and right principal ideal domain and there is a right division algorithm in $R$: for all
$g,f\in R$, $g\not=0$, there exist unique $r,q\in R$ with ${\rm
deg}(r)<{\rm deg}(f)$, such that $g=qf+r.$ There is also a left division algorithm in $R$ \cite[p.~3 and Prop. 1.1.14]{J96}. (Our
terminology is the one used by Petit \cite{P66};  Jacobson's is vice versa.)

 Two non-zero elements $f,g\in R$ are called \emph{similar}
 ($f\thicksim g$) if and only if there exist $h,q,u\in R$ such that
 $$1=hf+qg \text{ and } u'f=gu$$
for some $u'\in R$. Equivalently, $f$ and $g$ are similar if $R/Rf$ and $R/Rg$ are isomorphic as $R$-modules \cite[p.~11]{J96}.
Obviously, $f\thicksim g$ implies that ${\rm deg}(f)={\rm deg} (g)$.

\subsection{The characteristic $p>0$ case}

Let $K$ be a field of characteristic $p$ and $R=K[t;\delta]$, then
$$(t-b)^p=t^p-V_p(b), \quad V_p(b)=b^p+\delta^{p-1}(b),\quad (t-b)^{p^e}=t^{p^e}-V_{p^e}(b)$$
 for all $b\in K$
 with $V_{p^e}(b)=V_p^e(b)=V_p(\dots (V_p(b))\dots )$ \cite[p.~17ff]{J96}.
For any $p$-polynomial
$$f(t)=a_0t^{p^e}+a_1t^{p^{e-1}}+\dots+a_et+d\in D[t;\delta]$$
we thus have
$$f(t)-f(t-b)=a_0V_{p^e}(b)+a_1V_{p^{e-1}}(b)+\dots+a_eb$$
for all $b\in K$ and define
$$V_f(b)=a_0V_{p^e}(b)+a_1V_{p^{e-1}}(b)+\dots+a_eb.$$

\subsection{Nonassociative algebras obtained from differential polynomial rings} \label{subsec:structure}

Let $K$ be a field and $f\in R=K[t;\delta]$ of degree $m$. Let ${\rm mod}_r f$ denote the remainder of right division by $f$.
Define  $F={\rm Cent}(\delta)=\{a\in K\,|\, \delta(a)=0\}$.

 \begin{definition} (cf. \cite[(7)]{P66})
  The vector space
$$R_m=\{g\in K[t;\delta]\,|\, {\rm deg}(g)<m\}$$
 together with the multiplication
 $$g\circ h=gh \,\,{\rm mod}_r f $$
 is a unital nonassociative algebra $S_f=(R_m,\circ)$ over
 $$F_0=\{a\in K\,|\, ah=ha \text{ for all } h\in S_f\}.$$
\end{definition}

 $F_0$  is a subfield of $K$ \cite[(7)]{P66} and it is easy to check that $F_0={\rm Cent}(\delta)$.
 The algebra $S_f$ is also denoted by $R/Rf$ \cite{P66, P68}
 if we want to make clear which ring $R$ is involved in the construction.
 In the following, we call the algebras $S_f$  \emph{Petit algebras} and denote their multiplication simply
 by juxtaposition. Without loss of generality, we may assume that $f$ is monic, since $S_f= S_g$ for all
 $g=af$ with $a\in K^\times$.

  Using left division by $f$ and
  the remainder ${\rm mod}_l f$ of left division by $f$ instead, we can define the multiplication for
  another unital nonassociative algebra on $R_m$ over $F$, called $\,_fS$ or $R/fR$.
We will only consider the Petit algebras $S_f$, however, since every algebra
$\,_fS$ is the opposite algebra of some Petit algebra (cf. \cite[(1)]{P66}).

 Right multiplication with $0\not=g\in S_f$
is given by $R_g:S_f\longrightarrow S_f,$ $h\mapsto hg$, and
is a left $K$-module endomorphism. Left multiplication
$L_g:S_f\longrightarrow S_f,$ $h\mapsto gh$  is an $F$-module endomorphism \cite{P66}, and if we view $S_f$ as a
right module over ${\rm Nuc}_r(S_f)$, a right ${\rm Nuc}_r(S_f)$-module endomorphism.

Clearly $S_f$ has no zero divisors if and only if $R_g$ and $L_g$ are injective.

\begin{theorem} (cf. \cite[(2), p.~13-03, (5), (6), (7), (9), (14)]{P66})
\label{thm:structure}
Let $f
\in R = K[t;\delta]$.
\\
(i) If $S_f$ is not associative then ${\rm Nuc}_l(S_f)={\rm Nuc}_m(S_f)=K$ and
$${\rm Nuc}_r(S_f)=\{g\in R_m\,|\, fg\in Rf\}.$$
The right nucleus of $S_f$ is Amitsur's invariant ring of $f$.
\\ (ii) The powers of $t$ are associative if and only if $t^mt=tt^m$
 if and only if $t\in {\rm Nuc}_r(S_f)$ if and only if $ft\in Rf.$
\\ (iii) If $f$ is irreducible then ${\rm Nuc}_r(S_f)$ is an associative division algebra.
\\ (iv) Let $f\in R$ be irreducible and $S_f$ a finite-dimensional $F$-vector space
or free of finite rank as a right ${\rm Nuc}_r(S_f)$-module. Then $S_f$
is a division algebra.
\\ Conversely, if $S_f$ is a division algebra then $f$ is irreducible.
\\
(v) $S_f$ is associative if and only if $f$ is a two-sided element (i.e., generates a two-sided ideal $Rf$).
In that case, $S_f$ is the usual quotient algebra $K[t;\delta]/(f)$.
\\ (vi)  $f$ is irreducible if and only if $S_f$ is a right division algebra over $F$
(i.e., each non-zero element in $S_f$ has a left inverse: there is $z\in S_f$ such that $zh=1$), if and only if
 $S_f$ has no zero divisors.
\end{theorem}

Recall that a polynomial $f\in R= K[t;\delta]$ is \emph{bounded} if there exists $0\not=f^*\in R$, such that $Rf^*=f^*R$
is the largest two-sided ideal of $R$ contained in $Rf$.

 If $f\in R$ is bounded then $f$ is irreducible if and only if
${\rm Nuc}_r(S_f)$ has no zero divisors if and only if ${\rm Nuc}_r(S_f)$ is an associative division algebra
(cf. \cite[Proposition 4]{G} which sums up classical results from \cite{J43}).
\cite[Theorem 4]{C} yields:

\begin{theorem} \label{thm:bounded}
Let  $f\in R$ be irreducible. Then $f$ is bounded if and only if
$S_f$ is free of finite rank as a $ {\rm Nuc}_r(S_f)$-module.
In this case, $S_f$ is a division algebra.
\end{theorem}

\begin{proof}
The first part of the statement is \cite[Theorem 4]{C}.
Since $f$ irreducible, $S_f$ is a right division algebra and $L_h$ is injective for all  $h\in S_f$,
$h\not=0$, as observed in \cite[Section 2., (7)]{P66}.  The second part then follows from the fact that $S_f$ is free of
finite rank as a $ {\rm Nuc}_r(S_f)$-module,
which means the injective  $ {\rm Nuc}_r(S_f)$-linear map $L_h$ is also surjective.
\end{proof}

$R= K[t;\delta]$ has finite rank over its center if and only if
$K$ is of finite rank over $C_t=\{a\in K\,|\, at=ta\}$ if and only if all polynomials of $R$ are bounded and
if for all $f$ of degree non-zero, ${\rm deg}(f^*)/{\rm deg}(f)$ is bounded in $\mathbb{Q}$ ($f^*$ being the bound of
$f$) \cite[Theorem IV]{CI}. Since here $C_t= {\rm Const}(\delta)=F$, we conclude:

\begin{proposition} \label{prop:important}
Assume that one of the two following equivalent conditions hold:
\\ (i) $R=K[t;\delta]$ has finite rank over its center;
\\ (ii) $K/F$ is a finite field extension.
\\ Then every $f\in R$ is bounded. In particular, if $f$ is irreducible then $S_f$ is a division algebra.
\end{proposition}

Note that if $K/F$ is a finite field extension then the derivation $\delta$ is trivial, or $K$ has characteristic $p>0$.

We will assume throughout the paper that $f\in K[t;\delta]$ has $\text{deg}(f) = m \geq
2$ (if $f$ has degree $m=1$ then $S_f\cong K$) and that
$\delta\not=0$. Without loss of generality, we could only  look at monic $f$, but will do so only when
explicitly mentioned.

%
%

\section{Nonassociative algebras whose right nucleus is a central simple algebra}

 We use the terminology from \cite{Am} with the only exception that that in our definition
 of $K[t;\delta]$, we look at polynomials with the coefficients
 written on the left, not on the right-hand-side as in \cite{Am}. All results, however, work analogously in this case.

  By \cite[Theorem 4.2]{N}, given a field extension $K/F$ in characteristic zero,
 $F$ is the field of constants of a derivation of $K$ if and only if $F$ is algebraically closed in $K$.

 In this section, let $K$ be a field of characteristic 0.
 Let $\delta$ be a derivation of $K$ with $F={\rm Const}(\delta)$ and
  $f\in R=K[t;\delta]$. The finite-dimensional
 associative $F$-algebra ${\rm Nuc}_r(S_f)$ is called the \emph{invariant ring} of $f$ by Amitsur  \cite[p.~260]{Am},
  in recent literature it is also referred to as the \emph{eigenspace} of $f$.

Let $V$ be an $K$-vector space. An additive map $T:V\longrightarrow V$, such that
$T(\alpha v)=\alpha T(v)+\delta(\alpha)v$
for all $v\in V$ and $\alpha\in K$, is called a \emph{pseudo-linear transformation} on $V$.
Given a basis of $V$, a pseudo-linear transformation $T$ on $V$ is given by a matrix. Moreover,
$(V,T)$ is isomorphic to $K[t;\delta]/f(t)K[t;\delta]$ for some $f(t)\in K[t;\delta]$ which is called the \emph{characteristic polynomial}
of $T$ \cite[p.~250]{Am}. The characteristic polynomial is uniquely determined up to similarity
and any polynomial $f(t)$ is the characteristic polynomial of some pseudo-linear transformation $(V,T)$
(simply define $V=K[t;\delta]/K[t;\delta]f(t)$ and $T(p(t)+K[t;\delta]f(t))=t p(t)+K[t;\delta]f(t)$).

Let $(V,T)$ and $(V',T')$ be two pseudo-linear transformations with characteristic polynomials  $f, g\in K[t;\delta]$
where ${\rm deg}(f)=m$ and
 ${\rm deg}(g)=n$. Then there is a pseudo-linear transformation $T\times T'$ on the tensor product $V\otimes V'$ defined via
 $$(T\times T')(u)=\sum_i T(v_i) \otimes w_i + \sum_i v_i\otimes T'(w_i)$$
  for all $u=\sum_i v_i\otimes w_i\in V\otimes V'$.

  Furthermore, let $f,g\in K[t;\delta]$ where ${\rm deg}(f)=m$ and
 ${\rm deg}(g)=n$, and $T$ and $T'$ be the pseudo-linear transformation defined using $f$ and $g$.
 Then the \emph{resultant} $f\times g$ of $f$ and $g$ is any characteristic polynomial of
$T\times T'$,
so that $f\times g$
is a polynomial of degree $nm$
uniquely determined up to similarity \cite[p.~255]{Am}.

 A differential polynomial $f\in K[t;\delta]$ of degree $m$ is called an \emph{$A$-polynomial} if there is some
 $\widetilde{f}\in K[t;\delta]$ of degree $n$ such that the resultant $f\times  \widetilde{f}$ is similar to
 $e_{mn}$,  the characteristic polynomial of the pseudo-linear transformation corresponding to the zero  $mn\times mn$
 matrix \cite[p.~263]{Am}.

 Amitsur's results  tell us when ${\rm Nuc}_r(S_f)$ is a central simple algebra:

\begin{theorem} \cite[Lemma  17, 18, 19, Theorem 17, Corollary, Lemma  22]{Am}\label{thm:general}
Let  $f, g\in K[t;\delta]$ with ${\rm deg}(f)=m\geq 2$ and
 ${\rm deg}(g)=n\geq 2$.
\\ (i)  ${\rm Nuc}_r(S_f)$ has dimension $m^2$ if and only if $f$ is an $A$-polynomial.
\\ (ii)  If $f$ is an $A$-polynomial then ${\rm Nuc}_r(S_f)$ is a central simple algebra of degree
$m$ which is split by
$K$.
\\ (iii) If $f$ and $g$ are $A$-polynomials then so is $h=f\times g$ and
$${\rm Nuc}_r(S_h)={\rm Nuc}_r(S_f)\otimes_F {\rm Nuc}_r(S_g).$$
 (iv) If $f$ and $g$ are $A$-polynomials then
$${\rm Nuc}_r(S_f)\cong {\rm Nuc}_r(S_g)$$
 if and only if $f\sim g(t+a)\sim g(t)\times t+a$ for some $a\in K$. In particular,
 $$S_f\cong S_g \text{ implies that }f\sim g(t+a)\sim g(t)\times t+a$$
  for some $a\in K$.
\\ (v) Suppose $f$ is an $A$-polynomial. Then
$${\rm Nuc}_r(S_f)\cong {\rm Mat}_m(F)$$
if and only if one of the following holds:
\begin{itemize}
\item $f\sim e_m\times t+c$ for some $c\in K$;
\item $f$ decomposes into irreducible factors and at least one factor is linear of the form $t+c$ for some $c\in K$
(then $f\sim e_m\times t+c$).
\\ In particular, then the irreducible factors of $f$  are all similar to $t+c$.
\end{itemize}
\end{theorem}

 Let $L/K$ be a field extension such that $\delta$ extends to
  $L$. Then $L[t;\delta]$ is an Ore extension of $K[t;\delta]$ and the constant field  $F={\rm Const}(\delta|_K)$
  of $\delta=\delta|_K$ is contained in the constant field $C={\rm Const}(\delta)$.
  If $L=K\cdot C$ is the composite field of $K$ and $C$, we say $L$ is a \emph{constant extension} of $K$.
  It is clear that for $f\in K[t;\delta]$,
  $${\rm Nuc}_r(K[t;\delta]/K[t;\delta]f)\subset {\rm Nuc}_r(L[t;\delta]/L[t;\delta]f).$$

\begin{theorem}
Let  $f\in K[t;\delta]$ be of degree $m$ and $L/K$ a field extension such that $\delta$ extends to
  $L$ and $C={\rm Const}(\delta)$. Suppose that $L$ is a constant extension of $K$.
\\ (i) If $f$ is an $A$-polynomial then $f\in L[t;\delta]$ is an $A$-polynomial and
$${\rm Nuc}_r(L[t;\delta]/L[t;\delta]f)\cong {\rm Nuc}_r(K[t;\delta]/K[t;\delta]f) \otimes_F C. $$
 (ii) Suppose
  $B={\rm Nuc}_r(S_f)$ is a central simple algebra of degree $m$ over $F$ with
 $f\in K[t;\delta]$. Then $C$ splits $B$ if and only if $f$ has a left or right root in $L$,
 i.e. $f=(t-a)g(t)\in L[t;\delta]$ or $f=g(t)(t-a)\in L[t;\delta]$.

 In particular, then $S_f\otimes_F C$ has right nucleus ${\rm Mat}_m(C)$.
\end{theorem}

This follows from \cite[Theorem 20]{Am} and \cite[Corollary, p.~270]{Am}.

\begin{remark}\label{re:6}
Since every automorphism  of a nonassociative algebra maps the right nucleus onto itself,
for every $A$-polynomial $f$ which is not two-sided, each $H\in {\rm Aut}_F(S_f)$ satisfies
$H|_B\in {\rm Aut}_F(B)$
when restricted to the central simple algebra $B={\rm Nuc}_r(S_f)$, thus $H|_B$ is an
inner automorphism of $B$. By an analogous argument, also $H|_K\in {\rm Aut}_F(K)$.
\end{remark}

%
%

\section{Algebras whose right nucleus is split by an extension in which $F$ is algebraically closed}

 Let $F$ be a field of characteristic 0.

\begin{theorem} \cite[Lemma 20]{Am}\label{thm:lemma20}
 (i) Every central simple algebra $B$ of degree $m$ over $F$ which is split by a field extension
$K/F$ in which $F$ is algebraically closed, is isomorphic to ${\rm Nuc}_r(S_f)$ for some $f\in K[t;\delta]$ of degree $m$ and a suitable $\delta$ with
$F={\rm Const}(\delta)$. The differential polynomial $f$ is an $A$-polynomial.
\\ (ii) Every central division algebra $D$ of degree $m$ over $F$ is isomorphic to ${\rm Nuc}_r(S_f)$
for some $f\in K[t;\delta]$ of degree $m$ and a suitable differential field $(K,\delta)$.
\end{theorem}

Note that (ii) follows from (i), since for every central division algebra $D$ over $F$,
the function field $K(X)$ of the Severi-Brauer variety $X$ of $D$ splits $D$  (\cite[p.~245]{Am} or \cite{Am3}),
 and we can always find a derivation $\delta$ on $K(X)$ with $F={\rm Const}(\delta)$,
 as $F$ is algebraically closed in $K(X)$.

As an immediate consequence of  Theorem \ref{thm:lemma20} and Remark \ref{re:6}, we now get the following results:

\begin{theorem} \label{thm:mainI}
 For every central simple algebra $B$ of degree $m$ over $F$  which
is split by a field extension $K/F$ in which $F$ is algebraically closed,
there is a
derivation $\delta$ on $K$ with field of constants $F$ and a differential polynomial $f\in K[t;\delta]$ of degree $m$,  such that
$$S_f=K[t;\delta]/K[t;\delta]f$$
is an infinite-dimensional algebra over $F$ with right nucleus $B$ and left and middle nucleus $K$.
Every automorphism $H\in {\rm Aut}_F(S_f)$ extends an inner automorphism of $B$.
\end{theorem}

We conclude from \cite[p.~246]{Am}:

\begin{corollary}  \label{cor:mainI}
 For every central division algebra $D$ of degree $m$ over $F$,
 there exists a field extension $K/F$ in which $F$ is algebraically closed,
a derivation $\delta$ on $K$ with field of constants $F$, and a differential polynomial $f\in K[t;\delta]$ of degree $m$,
 such that
$$S_f=K[t;\delta]/K[t;\delta]f$$
is an infinite-dimensional  algebra over $F$ with right nucleus $D$, and left and middle nucleus $K$.
$K$ splits $D$ and every automorphism $H\in {\rm Aut}_F(S_f)$
  extends an inner automorphism of $D$.
\end{corollary}

The fact that $D$ is a division algebra does not imply that $f$ is irreducible, so $S_f$ might not be
a right division algebra.

\begin{corollary}
If the differential polynomial $f$ in Corollary \ref{cor:mainI} is irreducible, then $S_f$ is an infinite-dimensional right division algebra over $F$
and therefore does not have zero divisors.
\end{corollary}

If $f$ is an irreducible $A$-polynomial, it is not bounded by Theorem \ref{thm:bounded}.

\begin{example}
Suppose $F=\mathbb{R}$. The only central division algebra over $\mathbb{R}$ is $D=(-1,-1)_\mathbb{R}$.
The function field $K$ of the projective real conic given by $x^2+y^2+z^2=0$ is
a  field extension of $\mathbb{R}$ in which $\mathbb{R}$ is algebraically closed and that splits $D$.
There exists a derivation $\delta$ on $K$ with $\mathbb{R}={\rm Const}(\delta)$.
Thus there is an $A$-polynomial $f\in K[t;\delta]$ of degree $2$, such that
$$S_f=K[t;\delta]/K[t;\delta]f=K\oplus Kt$$
is an infinite-dimensional unital algebra over $\mathbb{R}$ with right nucleus $(-1,-1)_\mathbb{R}$, and left
and middle nucleus $K$.

For $B={\rm Mat}_m(\mathbb{R})$ and  any field
 extension $K'$ of $\mathbb{R}$ in which $\mathbb{R}$ is algebraically closed,
with a derivation $\delta$ on $K'$ such that $\mathbb{R}={\rm Const}(\delta)$,
 there is a reducible $A$-polynomial $f\in K'[t;\delta]$ of degree $m$, such that
$$S_f=K'[t;\delta]/K'[t;\delta]f$$
is an infinite-dimensional unital algebra over $\mathbb{R}$ with right nucleus $B$ and left
and middle nucleus $K'$.
\end{example}

%
%

\section{Algebras whose right nucleus is a $p$-algebra}

Let now $K$ be a field of characteristic $p>0$ together with a derivation $\delta$ on $K$. Put $R=K[t;\delta]$
and $F={\rm Const}(\delta)$.
 There are two cases which can occur: either $\delta$ is an algebraic derivation, or $\delta$ is transcendental
 which means $[K:F]=\infty$.
 We assume that $\delta$ is an algebraic derivation of degree $p^e$ with minimum polynomial
$$g(t)=t^{p^e}+c_1t^{p^{e-1}}+\dots+ c_et\in F[t]$$
 of degree $p^e$. Then $K=F(u_1, \dots,u_e)=F(u_1)\otimes_F\dots \otimes_F F(u_e)$ with $u_i^p=a_i\in F$
 for all $i\in\{1,\dots,e\}$, and $[K:F]=p^e$,
 that is $K$ is a finite purely inseparable field extension of exponent one and
$K^p\subset F\subset K$.
 The center $C(R)$ of $R$ is $F[z]$ with $z=g(t)-d_0$, $d_0\in F$,
 and the two-sided elements in $R$ have the form $uh(t)$ with $u\in K^\times$, $h(t)\in C(R)$.

Recall that a central simple algebra $B={\rm Mat}_r(D)$ over a field $F$ of characteristic $p$ is a \emph{$p$-algebra}
 if it has index
 $p^n$, equivalently, if
  its exponent is a power of $p$ \cite[p. 154]{J96}.

Note that for $f(t)=g(t)-d\in F[t]$ (so $f(t)$ is two-sided in this case),
$$(K,\delta,d)=K[t;\delta]/K[t;\delta] f(t)$$
is an associative central simple $F$-algebra called a \emph{differential extension of $K$} and treated in \cite[p.~23]{J96}.
  $K$ is a maximal subfield of $(K,\delta,d)$.

\begin{theorem} \cite[Lemma 20']{Am}\label{thm:lemma20'}
 Let $B$ be a $p$-algebra of degree $m$ over $F$ which is split by a
purely inseparable extension $K$ of exponent one (i.e., has exponent $p$), such that $m\leq [K:F]$.
Then $$B\cong{\rm Nuc}_r(S_f)$$
 for some $f\in K[t;\delta]$ of degree $m$ and a suitable $\delta$ with
$F={\rm Const}(\delta)$.
\end{theorem}

We start by looking at the case that $m=[K:F]=p^e$ and
immediately obtain (i) and (ii) in the following result on $p$-algebras by employing only
Theorem \ref{thm:structure} (v) from
Petit \cite{P66} and Amitsur's Theorem \ref{thm:lemma20'} (only the fact that then $B$ is cyclic
uses Hood's Main Theorem \cite[Main Theorem]{H}):

 \begin{theorem} \label{thm:mammone}
Let $B$ be a $p$-algebra of degree $m$ over $F$  which is split by a purely inseparable field extension $K$ of exponent one
with $m=[K:F]$.
\\ (i) There is an algebraic derivation $\delta$ on $K$  of degree $m$ with  minimum polynomial $g(t)$ such that
the
center  of $K[t;\delta]$ is $F[z]$ with $z=g(t)-d_0$, $d_0\in F$,
and
$$B=(K,\delta,d_0)$$
with $f(t)=g(t)-d_0$. $B$ is a cyclic algebra.
\\ (ii) $F[t]/(f)$ is a subfield of $B$ of degree $p^e$ over $F$ if and only if $f$ is irreducible in $F[t]$.
\\ (iii)  $f\in K[t;\delta]$ is irreducible if and only if $B$ is a division algebra.
\\ (iv) $B\cong {\rm Mat}_{p^e}(F)$ if and only if there is $b\in K$ such that
$$d_0=V_g(b)=V_{p^e}(b)+c_1V_{p^{e-1}}(b)+\dots+c_eb.$$
\end{theorem}

\begin{proof}
  (i) If $m=[K:F]$ then there is a differential  polynomial $f\in K[t;\delta]$ of degree $m$ and a suitable
 $\delta$ such that $B\cong{\rm Nuc}_r(S_f)$ by Theorem \ref{thm:lemma20'}. Here $B$ is an associative
  subalgebra of $S_f$ of dimension $m^2$ and $S_f$ has dimension $m^2$ as well. Therefore $S_f=B$ is associative and
  $f\in K[t;\delta]$ must be a two-sided differential  polynomial of degree $m$, i.e.
  $B=K[t;\delta]/(f)$ is a quotient algebra (Theorem \ref{thm:structure} (v)). Without loss of generality
   we may assume $f$ is monic. Thus $f\in C(R)$ and since $f$ has degree $m=p^e$, we obtain that
  $f(t)=g(t)-d_0$ and so $B=(K,\delta,d_0)$.
  $K$ is a purely inseparable field extension of $F$ which is an (even maximal) subfield of $B$ splitting $B$,
  therefore $B$ is cyclic \cite[Main Theorem]{H}.
\\ (ii)  Since here $f(t)\in F[t]$, we know that $F[t]/(f)$ is a subfield of $B$ of degree $p^e$ over $F$ if and only if $f$ is
  irreducible in $F[t]$.
  \\ (iii) is  \cite[Proposition 4]{G} and (iv) is a consequence from (i) together with
  Theorem \cite[Theorem 1.3.27]{J96}.
  \end{proof}

  \begin{remark} Let us briefly put the previous result into context:
 \\ (i) Let A be a central simple $p$-algebra of degree $p^n$ over $F$. It is a well known classical result
 that $A$ is cyclic over $F$ if and only if
$A$ has a subfield $K$ such that $K$ is a  purely inseparable extension of $F$ and $K$ is a splitting field for A
(this is \cite[Main Theorem]{H}, which removed Albert's restriction that $K$ be simple from \cite[Theorem (7.27)]{Al}).
  \\ (ii)
Mammone  characterized the central simple algebras split by a purely inseparable field extension $K$
of exponent one in \cite{M}: in particular, if $B$ is a central simple algebra over $F$
of degree $m=p^e$ containing $K$ where $[K:F]=m$, then $B$ is a \emph{differential crossed product}, that means
$B$ contains a $K$-basis of the form $\{z_1^{i_1}\cdots z_n^{i_n}\,|\, 0\leq i_k\leq p-1\}$
 satisfying a kind of commutativity law with elements of $K$ which involves a set of $n$ $F$-derivations of $K$.  The algebra $B$
  then yields elements $b_i=z_i^p$ and $u_{ij}=z_iz_j-z_jz_i$ in $K$. Conversely, given sets $B=\{b_i\,|\,i=1,\dots,n\}$
  and
  $U=\{u_{ij}:i,j=1,\dots,n\} $ satisfying certain relations involving  $F$-derivations of $K$, then
  $(U,B)$ arises from such a differential crossed product.
\end{remark}

In case $m<[K:F]=p^e$ we obtain a nonassociative algebra of dimension $mp^e$ containing $B$ as right nucleus:

 \begin{theorem}  \label{thm:mammoneII}
Let $B$ be a $p$-algebra of degree $m$ over $F$  which
is split by a purely inseparable extension $K$ of exponent one
such that $m<[K:F]$.
\\ (i) There is an algebraic derivation $\delta$ and a differential polynomial $f\in K[t;\delta]$ of degree $m$ such that
 $$S_f=K[t;\delta]/K[t;\delta]f$$
 is an algebra over $F$ of dimension $mp^e$ with  right nucleus $B$, left and middle nucleus $K$, and nucleus
 ${\rm Nuc}(S_f)=B\cap K$ an intermediate field of $K/F$, unequal to $K$.
\\ (ii) $f$ is irreducible if and only if $B$ is a division algebra, if and only if $S_f$ is a division algebra.
\\ (iii)  Every automorphism $H\in {\rm Aut}_F(S_f)$  extends an inner automorphism of $B$ and an automorphism of $K$.
\end{theorem}

\begin{proof}
(i) The existence of a suitable $f$ follows from Theorem \ref{thm:lemma20'} and the statements on the left and middle nuclei
 from Theorem \ref{thm:structure}. Since $f$ is not two-sided,  $K$ is not contained in the right nucleus of $S_f$, i.e. not contained in $B$
\cite[Theorem 9]{Pu16}. Thus ${\rm Nuc}(S_f)=B\cap K$ is properly contained in $K$, so that it is an intermediate field
of the field extension $K/F$.
\\ (ii)
 By Proposition \ref{prop:important} and Theorem \ref{thm:structure}, $f$ is irreducible  if and only if $B$ is a division algebra, if and only if $S_f$
 is a division algebra.
\\ (iii) An automorphism of $S_f$ extends both an
inner automorphism of $B$ and an automorphism of $K$ by Remark \ref{re:6}.
\end{proof}

 \begin{corollary}\label{cor:II}
Let  $D$ be a division $p$-algebra of degree $m$
over $F$  which is split by a purely inseparable extension $K$ of exponent one such that $m<[K:F]$.
Then there is an irreducible polynomial $f\in K[t;\delta]$ of degree $m$ such that
   $S_f$ is a division algebra over $F$ of dimension $m p^e$  with right nucleus $D$, left and middle nucleus $K$,
   and nucleus $D\cap K$ an intermediate field of $K/F$, unequal to $K$.
\end{corollary}

The fact that $f$ is irreducible  in Corollary \ref{cor:II} follows from Proposition \ref{prop:important}.
Note that every division $p$-algebra over $F$ split by $K$ has degree $m\leq [K:F]$, so that
Theorem \ref{thm:mammone} (iii) and Corollary \ref{cor:II} cover all possible cases for a division $p$-algebra.

We could ask for the algebra $S_f$ of smallest possible dimension which contains a given central simple algebra
 $B$ as a right nucleus.
This is equivalent to asking for a purely inseparable extension $K$ of exponent one splitting $B$ of smallest
 possible degree $[K:F]=p^e$ satisfying $m<[K:F]$, which in turn is connected to the question how many cyclic algebras
 are needed when saying that $B$ is similar to a product of cyclic algebras of degree $p$ in the Brauer group $Br(F)$.

 \begin{theorem}
Let $B$ be a  $p$-algebra  over $F$ of degree $m$, index $d=p^n$ and exponent $p$, such that
$m=r^2p^n< p^{d-1}$. Then there is
 a purely inseparable extension $K$ of exponent one with $[K:F]=p^{d-1}$,
and a differential polynomial $f\in K[t;\delta]$ of degree $m$ such that
 $$S_f=K[t;\delta]/K[t;\delta]f$$
 is an algebra over $F$ of dimension $mp^{d-1}$ with  right nucleus $B$
 and the properties listed in Theorem \ref{thm:mammoneII}.
\end{theorem}

\begin{proof}
Let $B$ be a  $p$-algebra of index $p^n$ and exponent $p$. Then there is a purely inseparable field extension $K/F$
of exponent one with $K=F(u_1, \dots,u_{d-1})$, $u_i^p=a_i\in F$, and $[K:F]=p^{d-1}$, which splits $B$
\cite[Theorem 1.1.]{F}. We have $m=r^2d=r^2p^n$ for some $r\geq 1$.

We need $m=r^2p^n\leq [K:F]=p^{d-1}$ to be able to apply Theorem \ref{thm:lemma20'}.
By Theorem \ref{thm:lemma20'} this implies that
$B\cong{\rm Nuc}_r(S_f)$
 for some $f\in K[t;\delta]$ of degree $m$ and a suitable $\delta$ with
$F={\rm Const}(\delta)$. Since each $f\in K[t;\delta]$ is bounded by Proposition \ref{prop:important},
 $B={\rm Nuc}_r(S_f)$ is a division algebra if and only if
$f$ is irreducible \cite[Proposition 4]{G}, if and only if $S_f$ is a
division algebra.
\end{proof}

We obtain that for a division algebra $D$, the smallest possible dimension $l$ of a division algebra $S_f$
containing $D$ as right nucleus  satisfies $m^2< l=mp^{e}\leq mp^{m-1}$:

 \begin{corollary} \label{cor:last}
Let $D$ be a division $p$-algebra of degree $m$ and exponent $p$ over $F$. Then there is
 a purely inseparable extension $K$ of exponent one with $[K:F]=p^{m-1}$,
and an irreducible differential polynomial $f\in K[t;\delta]$ of degree $m$ such that
 $$S_f=K[t;\delta]/K[t;\delta]f$$
 is a division algebra over $F$ of dimension $mp^{m-1}$ with  right nucleus $D$
 and the properties listed in Theorem \ref{thm:mammoneII}.
\end{corollary}

\begin{proof}
There is a purely inseparable field extension
 $K=F(u_1, \dots,u_{m-1})$ of exponent one, $u_i^p=a_i\in F$, and $[K:F]=p^{m-1}$, which splits $D$
\cite[Theorem 1.1.]{F}.

We need $m=p^n\leq [K:F]=p^{m-1}$ to be able to apply Theorem \ref{thm:lemma20'}. This
holds for all prime $p$ and $n\geq 1$ as it is equivalent to
$n\leq p^n-1$, i.e. to $n+1\leq p^n$, which is true for all prime $p$ and $n\geq 1$.
Therefore there is a purely inseparable field extension $K/F$
of exponent one with $m\leq [K:F]=p^{m-1}$ which splits $D$. By Theorem \ref{thm:lemma20'} this implies that
$B\cong{\rm Nuc}_r(S_f)$
 for some $f\in K[t;\delta]$ of degree $m$ and a suitable $\delta$ with
$F={\rm Const}(\delta)$. Since $D$ is a division algebra and $f$ bounded, $f$ is irreducible and $S_f$ is a division algebra.
\end{proof}


\end{document}